\documentclass[final, 12pt]{article}

\usepackage{amsmath}
\usepackage{amsfonts}
\usepackage{makeidx}
\usepackage{amsthm}
\usepackage{mathtools}
\usepackage{amssymb}
\usepackage{bm}
\usepackage{cite}
\usepackage{showkeys}
\usepackage[english]{babel}
\usepackage{dblaccnt}
\usepackage{accents}
\usepackage{graphicx}
\usepackage{psfrag}
\usepackage{subfig}
\usepackage{color}
\usepackage{float}
\usepackage{amscd}
\usepackage[mathscr]{euscript}
\usepackage{lipsum}
\usepackage{epstopdf}
\usepackage{algorithm}
\usepackage{algpseudocode}

\newcommand{\N}{\ensuremath{\mathbb{N}}}

\newcommand{\R}{\ensuremath{\mathbb{R}}}
\newcommand{\C}{\ensuremath{\mathbb{C}}}

\renewcommand{\d}{\mathrm{d}}
\newcommand{\di}{\mathbf{d}}
\renewcommand{\O}{\mathcal{O}}
\newcommand{\ol}{\overline}
\renewcommand{\i}{\mathrm{i}}

\newcommand{\p}{\mathbf{p}}
\newcommand{\q}{\mathbf{q}}

\renewcommand{\Re}{\mathrm{Re}\,}
\renewcommand{\Im}{\mathrm{Im}\,}

\newcommand{\eps}{\varepsilon}
\newcommand{\epsr}{\varepsilon_\mathrm{r}}

\renewcommand{\S}{\mathbb{S}}

\renewcommand{\H}{{\mathbf{H}}}

\newcommand{\E}{{\mathbf{E}}}
\newcommand{\Ei}{{\mathbf{E}_\mathrm{in}}}

\newcommand{\curl}{\mathrm{curl}\,}

\newcommand{\x}{\mathbf{x}}
\newcommand{\y}{\mathbf{y}}
\newcommand{\z}{\mathbf{z}}
\newcommand{\g}{\mathbf{g}}
\newcommand{\f}{\mathbf{f}}
\renewcommand{\u}{\mathbf{u}}
\renewcommand{\v}{\mathbf{v}}

\newtheorem{defi}{Definition}
\newtheorem{lemma}[defi]{Lemma}
\newtheorem{theorem}[defi]{Theorem}

\newtheorem{assumption}[defi]{Assumption}

\begin{document}

%\title{The Direct and Generalized Linear Sampling Methods for Maxwell's Equations}
\title{Orthogonality Sampling Method for the Electromagnetic Inverse  Scattering Problem}

\author{Isaac Harris\thanks{Department of Mathematics, Purdue University, West Lafayette, IN 47907; (\texttt{harri814@purdue.edu})} \and Dinh-Liem Nguyen\thanks{Department of Mathematics, Kansas State University, Manhattan, KS 66506; (\texttt{dlnguyen@ksu.edu})}
}

\date{}
\maketitle

\begin{abstract}
This paper is concerned with the electromagnetic inverse scattering problem that aims to determine 
the location and shape of  anisotropic  scatterers from far field data (at a fixed frequency). 
We study the orthogonality  sampling method which is a simple, fast  and robust imaging method  for
solving  the electromagnetic inverse shape problem.
We first  provide a theoretical  foundation for the  sampling method and a resolution 
analysis of its imaging functional. We then establish an equivalent relation between the orthogonality 
sampling method and direct sampling method as well as  resolution analysis for the latter.  
The analysis uses the Factorization Method for the far field operator and it plays an  
important role in the justifications along with the Funk-Hecke integral identity.  
Finally, we present some numerical examples to validate 
the performance of the sampling methods for anisotropic scatterers in three dimensions. 
\end{abstract}

\sloppy


\bigskip


%%%%%%%%%%%%%%%%%%%%%%%%%%%%%%%%%%%%%%%%%%%%%%%%%%%%%%%%%%%%%%%%%%%%
\section{Introduction}
In this paper, we consider the inverse shape problem that is derived from the time-harmonic electromagnetic scattering of an inhomogeneous anisotropic medium. In many physical applications such as non-destructive testing and medical imaging one wishes to infer the shape and/or material properties of the scatterer from measured electromagnetic data. We assume that the far field measurements are known and we wish to analyze two sampling methods for recovering the scatterer. Sampling methods generally fall under the category of qualitative (otherwise known as non-iterative or direct) reconstruction techniques. These methods are advantageous to use since they require little a-prior information to implement and are computationally simple. Qualitative methods have been used to solve multiple inverse shape problems in electromagnetic scattering (see for e.g. \cite{Cakon2011, Kirsc2008, Hadda2014} and the references therein). These methods have also been extended to inverse problems in the time domain. In \cite{Chen2010, Cakon2019} the linear sampling and factorization   methods are extended to  inverse scattering problem problems in the time domain and in \cite{Cakon2017} the MUSIC algorithm is studied for recovering small volume scatterers for the time-dependent acoustic scattering problem.  Here we rigorously analyze both the orthogonality sampling method (OSM) and direct sampling method  (DSM) for recovering a penetrable inhomogeneous anisotropic medium from electromagnetic far field data.

In general, sampling methods allow one to recover the scatterer by connecting the scatterer to the solution of a linear ill-posed equation involving the far field operator. Roughly speaking, the linear sampling method gives that the so-called far field equation $\mathcal{F}g=\phi_{\y_s}$ is only solvable (via a regularization strategy) provided the sampling point is contained in the scatterer. Here the righthand side $\phi_{\y_s}$ is known and depends on the sampling point $\y_s \in \R^3$. This allows one to define an imaging functional that is the reciprocal of the norm of the solution to the far field equation which should only be non-zero as the regularization tends to zero for sampling points in the scatterer. See \cite{Cakon2011} for the analysis of the linear sampling method for the electromagnetic scattering problem. The factorization method gives that $\phi_{\y_s}$ is in the range of a positive self-adjoint compact operator defined by the far field operator if and only if the sampling point is in the scatterer. By appealing to Picard's criteria one can derive an imaging functional using the spectral decomposition of the far field operator see \cite{Kirsc2008}. 

The OSM was first introduced in~\cite{Potth2010} for the inverse acoustic scattering from sound soft scatterers.  Comparing with classical sampling methods the OSM is simpler to implement, can image (small) scatterers with only one incident field, and its stability can be easily justified. However, its mathematical foundation was
only partly known. For the Helmholtz equation case, the method  was rigorously justified in~\cite{Potth2010} for small scatterers and in~\cite{Liu2017} for scatterers with arbitrary shape using multi-static  data.   It was also first  proved in \cite{Liu2017} that the OSM is equivalent to the DSM studied in this cited paper via a remarkable connection to the analysis of the factorization method.   
 Recently, in \cite{Harri2019,Leem2018} this DSM was studied
 in connection to the spectral decomposition of the far field operator.  We also refer to~\cite{Gries2011} for the analysis of a multifrequency OSM. The OSM for Maxwell's equations and its analysis for the case of small scatterers have been recently established in~\cite{Nguye2019}. Motivated by these recent works we study in this paper the OSM for anisotropic Maxwell's equations and provide a  theoretical  foundation of the method for scatterers with arbitrary shape as well as a resolution analysis for its imaging functional. Furthermore, we  establish an equivalent relation between the OSM and the DSM as well as  resolution analysis for the latter. The factorization analysis for the far field operator 
as well as the Funk-Hecke formula play an important role in our theory. We also provide numerical results for three-dimensional anisotropic scatterers  to validate the efficiency of the sampling methods.  

We want to mention that the DSM in~\cite{Liu2017}  was extended to the electromagnetic case 
in~\cite{Alzaa2017}. This extension relies on the Factorization method analysis in~\cite{Kirsc2008}  for an isotropic medium and the decay rate of the imaging functional was not established.  Another DSM which is related to the OSM was studied in~\cite{Ito2013} for small electromagnetic scatterers. We also 
refer to~\cite{Hohag2001, Bao2009, Buhan2017} and references therein for results on  coefficient reconstruction  for the isotropic inverse electromagnetic scattering problem.

%. This imaging functional can be shown to be equivalent to an inner-product only requiring the knowledge of $\phi_{\y_s}$ and the far field operator. By appealing to the known factorization of the far field operator it can be shown that the DSM imaging functional is strictly positive and decays as the sampling point moves away from the scatterer. By appealing to the Funk-Hecke formula we are able to establish the decay rate for the proposed imaging functionals. Since the OSM and DSM only requires one to compute the norm or inner-product of known quantities this implies that  it computationally cheap to implement and stable in the presence of noise added to the far field data. The main contributions in this paper is the numerical investigation and theoretical validation of the sampling methods for anisotropic medium. 
%Here we also greatly reduce the regularity needed on a scalar-valued contrast  to prove the coercivity estimate. We will also consider the Generalized linear sampling method for our inverse problem which was first developed in \cite{GLSM}. Here we will again use the factorization of the far field operator to derive yet another imaging functional. The imaging functional here is derived from minimizing a residual in the far field equation with a penalty term that involves the inner-product of $\phi_{\y_s}$ and the far field operator. This method is motivated to rigorously connect the analysis of the Linear sampling and the Factorization methods. 

We now introduce some basic notations for the paper. 
Let $\O \subset \R^3$ be a domain (connected and open) with Lipschitz boundary, 
we indistinctly denote by $(\cdot,\cdot)$ the inner product of $L^2(\O)$ or $L^2(\O)^{3}$ and by $\|\cdot\|$ the associated norms. 
We further denote
\begin{align*}
H(\curl,\O) &= \big\{\v\in L^2(\O)^3: \curl \v \in L^2(\O)^{3} \big\}, \\   
H_{\mathrm{loc}}(\curl,\R^3) &= \big\{\v:\R^3\rightarrow\C^3: \v|_{B} \in H(\curl,B) \text{ for all ball } B \subset \R^3 \big\},
\end{align*}
where   $H(\curl,\O)$ is equipped by  usual inner product 
\[
(\cdot,\cdot)_{H(\curl,\O)} = (\curl\cdot,\curl\cdot) + (\cdot,\cdot).
\]
For the following sections we will first rigorously formulate the direct and inverse electromagnetic scattering problems under consideration in section 2. 
Section 3 is dedicated to an analysis of the far field operator and the factorization method which is necessary for the study of the OSM. 
We establish the main theoretical results of the paper in section 4. More precisely, we define the imaging functional for the OSM, prove 
its resolution and stability, and an equivalent relation to the imaging functional of the DSM. Lastly, numerical examples are given where we reconstruct bounded anisotropic scatterers in $\R^3$ using the electromagnetic far field data.  We see that the sampling methods are robust reconstruction methods that can recover scatterers of many different shapes and sizes. 

%%%%%%%%%%%%%%%%%%%%%%%%%%%%%%%%%%%%%%%%%%%%%%%%%%%%%%%%%%%%%%%%
\section{Direct and inverse problem formulation}

\label{sect:FP}

We consider the scattering of time-harmonic electromagnetic waves at  positive frequency
$\omega$ from a non-magnetic  inhomogeneous medium. Suppose that there is no free charge and current density. 
Then, the $H_{\mathrm{loc}}(\curl,\R^3)$ electric field $ \E$ and the magnetic field $\H$ satisfy the Maxwell's equations 
\begin{align}
\label{eq:Maxwell}
\curl \E - i \omega  \mu_0\H = 0 \quad \text{ and } \quad \curl \H +  i\omega \eps \E = 0 \quad \text{in } \R^3.
\end{align}
Here we assume that $\eps$ is the electric permittivity and  $\mu_0$ (positive constant) is the magnetic permeability 
of the medium. The permittivity $\eps$ is assumed to be a bounded  matrix-valued function. 
%\begin{figure}[h!]
%\centering
%    \psfrag{homo1}{\small $\,\, \xi \equiv 0I_3$}
%   \psfrag{homo}{\small $\eps \equiv \eps_0I_3, \mu \equiv \mu_0I_3$}
%   \psfrag{Es, Hs}{\small $\E - \E_\mathrm{in}, \H - \H_\mathrm{in}$}
%   \psfrag{Ei, Hi}{\small $\E_\mathrm{in}, \H_\mathrm{in}$}
%   \psfrag{xi}{\small$\xi(\x)$}
%   \psfrag{beta}{\small $ \quad \eps(\x), \mu(\x)$}
%    \psfrag{Omega}{\Large$\Omega$}
%   \includegraphics[width=7.5cm]{scattering_bianiso}
%   \caption{Schematic of the electromagnetic scattering from an inhomogeneous bi-anisotropic medium $\Omega$ characterized by 
%   matrix-valued functions $\eps(\x), \mu(\x), \xi(\x)$. }    
%   \label{scattering}
%\end{figure}
Let $\Omega$ be a bounded domain occupied by the non-magnetic inhomogeneous medium. The medium outside of 
$\Omega$ is assumed to be
 homogeneous. This means that there is a positive constant $\eps_0$   such that 
 $\eps = \eps_0I$ outside of $\Omega$, where  $I$ is the $3\times3$ identity matrix.
We define the relative material parameter and the wave number as
$$
\epsr = \eps/\eps_0, \quad
k = \omega\sqrt{\eps_0\mu_0}.
$$
Eliminating magnetic field $\H$ from~\eqref{eq:Maxwell} we obtain  
\begin{equation}
  \label{eq:Order2Total}
  \curl \curl \E   - k^2 \epsr  \E  = 0, \quad   \text{in } \R^3.
\end{equation}
The transmission conditions  across the boundary of $\Omega$ are given by
\begin{align}
\label{transmission}
\nu \times \E_+ = \nu \times \E_- \quad \text{ and}  \quad \nu \times \curl \E_+ = \nu \times \curl \E_-, \quad 
\text{on } \partial \Omega.
\end{align}
We denote by $\mathbf{F}_+$ and $\mathbf{F}_-$ the traces on $\partial \Omega$ from the exterior and interior of the domain $\Omega$  
for a vector-valued function $\mathbf F$ respectively, and $\nu$ is the unit outward normal vector on $\partial \Omega$.
Assume that we illuminate the inhomogeneous anisotropic medium with the electric and magnetic incident fields $ \E_\mathrm{in}$ and $\H_\mathrm{in}$, 
respectively, satisfying
\begin{align*}
 \curl \H_\mathrm{in} + i \omega \eps_0 \Ei = 0 \quad \text{ and} \quad \curl \Ei - i \omega\mu_0 \H_\mathrm{in} = 0, \quad \text{in } \R^3.
\end{align*}
Then  there arises the scattered electric field 
$\u$, defined by $\u:= \E - \Ei$. 
Since the incident field $\Ei$ satisfies the homogeneous Maxwell equation with wave number $k$ given by 
$$\curl\curl \Ei - k^2  \Ei= 0, \quad   \text{in } \R^3$$
 subtracting this equation from~\eqref{eq:Order2Total} we can conclude that 
the scattered field $\u \in H_{\mathrm{loc}}(\curl,\R^3)$ is the solution to
\begin{align}
\label{eq:secondOrder}
  \curl \curl \u - k^2 \epsr  \u =  k^2P\Ei  \quad  \text{in } \R^3,
\end{align}
where the contrast $P$ is defined by
\begin{equation*}
 P := \epsr - I.
\end{equation*}
Therefore, by definition we have that the support of the contrast $P$ is given by $\Omega$. Note that we also have the corresponding transmission conditions for the scattered field $\u$ following from~\eqref{transmission}.  
We complete the scattering problem by the  Silver-M\"{u}ller radiation condition for the scattered field $\u$ given by 
\begin{align}
 \label{eq:radiationCond}
 \curl \u \times \frac{\x}{|\x|} - ik \u = \mathcal{O}(|\x|^{-2}) \quad \text{as } |\x| \rightarrow \infty, 
\end{align}
which is assumed to hold uniformly with respect to $\x/|\x|$. 

 It is known (see for e.g. \cite{Monk2003a}) that \eqref{eq:secondOrder}--\eqref{eq:radiationCond} is well-posed provided that the contrast $P$ is bounded with non-negative real and imaginary parts with support $\Omega$ provided that the only solution to the homogeneous problem (i.e. $\Ei=0$) is trivial. We will assume that the homogeneous scattering problem only admits the trivial solution. This gives that the mapping $\Ei \mapsto \u$ is linear and bounded from $L^2(\Omega)^3$ into $H_{\mathrm{loc}}(\curl,\R^3)$.  
 
 Now we can talk about the inverse problem. To this end, we define $\widehat{\x} = \x/|\x|$,
 \begin{align*}
\mathbb{S}^2 = \{\x \in \R^3: |\x| = 1 \} \quad \text{ and } \quad L^2_t(\mathbb{S}^2)= \{ \v \in L^2(\mathbb{S}^2)^3: \widehat{\x} \cdot \v(\widehat{\x}) = 0,\,\, \widehat{\x} \in \mathbb{S}^2 \} .
\end{align*}
We consider the incident plane wave $\Ei(\x,\di, \q)=\q e^{ik\x\cdot\di}$, where  the vector $\di \in \S^2$ indicates the direction of the incident propagation 
and $\q \in \R^3$ is the polarization vector such that $\q \cdot \di =0$.  
It's well-known that we can express the corresponding scattered wave in terms of the asymptotic expansion 
$$
\u(\x,\di,\q) = \frac{e^{ik|\x|}}{ |\x|} \left ( \u^\infty(\widehat{\x},\di,\q) + O\left( \frac{1}{|\x|^2} \right) \right) \quad \text{ as } \quad |\x| \to \infty,
$$
uniformly in all observation directions $\widehat{\x} \in \mathbb{S}^2$. 
The function $\u^\infty(\widehat{\x},\di,\q)$ belonging to $L_t^2(\mathbb{S}^2)$ for each incident and observation direction is called the far field pattern. 

\textbf{Inverse problem:}
Determine the shape and location of the scatterer $\Omega$ given
 the far field pattern $\u^\infty(\widehat{\x},\di,\q)$  for all $\widehat{\x},\di \in \S^2$ for a single wave number.

%%%%%%%%%%%%%%%%%%%%%%%%%%%%%%%%%%%%%%%%%%%%%%%%%%%%%%%%%%%%%%%
\section{The far field operator and its factorization analysis}
\label{Factorization-section}

In this section, we will define and analyze the far field operator corresponding to \eqref{eq:secondOrder}--\eqref{eq:radiationCond}. The analysis in this section will be used to derive sampling methods to solve the inverse shape problem of recovering the scatterer $\Omega$ from the far field data. It is well-known that the far field pattern $\u^\infty(\widehat{\x},\di,\q)$  is linear in $\q$ and can be written as
$$
\u^\infty(\widehat{\x},\di,\q)  = \u^\infty(\widehat{\x},\di)\q 
$$
where $\u^\infty(\widehat{\x},\di)$ is a $3\times 3$ matrix and $\q \cdot \di = 0$, see~\cite{Colto2013}.  
The following reciprocity relation is important in our analysis and its  proof can be found in~\cite[Theorem 6.30]{Colto2013}.
\begin{theorem}
For all $\widehat{\x},\di \in \S^2$, the far field pattern $\u^\infty(\widehat{\x},\di)$ satisfies a reciprocity relation
$$
 \u^\infty(\widehat{\x},\di) = [ \u^\infty(-\di, -\widehat{\x}) ]^\top.
$$
\end{theorem}

We now define the far field operator $\mathcal{F}:L^2_t(\S^2) \to L^2_t(\S^2)$ as
$$
(\mathcal{F}\g)(\widehat{\x})  = \int_{\S^2}  \u^\infty(\widehat{\x},\di) \g(\di) \d s(\di). 
$$
In order to derive our sampling methods we will need to factorize the far field operator $\mathcal{F}$. To this end, it has been shown in \cite{Kirsc2004} that the far field operator has the following factorization $\mathcal{F} = G H $. Here the operator 
\begin{align}\label{operatorH}
H:L_t^2(\mathbb{S}^2) \to L^2(\Omega)^3 \,\, \text{ is given by } \,\, (H\g)(\x) =  \int_{\S^2} \g(\di) e^{ik\x\cdot\di} \d s(\di), \quad \x \in \Omega
\end{align}
and 
is the superposition of incident plane waves.  It is easy to see that $H$ is compact and injective. 
 The data to far field pattern operator
\begin{align}\label{operatorG}
G:  L^2(\Omega)^3 \to L_t^2(\mathbb{S}^2) \quad \text{ is given by } \quad G \f = \v^{\infty}
\end{align}
where $\v \in H_{\mathrm{loc}}(\curl,\R^3)$ is the solution to \eqref{eq:secondOrder}--\eqref{eq:radiationCond} with $\Ei = \f$. 
 The adjoint operator $H^*$ of $H$ is given by
$$
H^* \f (\di)= \di \times \left(\int_{\Omega}  \f(\x) e^{-ik \x \cdot \di} \, \d \x \right) \times \di, \quad  \di \in \S^2.
$$
From the form of the data to far field pattern operator a factorization of $G$ in \cite{Kirsc2004} gives that 
$G=H^* T$ with 
\begin{align}\label{operatorT}
T:L^2(\Omega)^3 \to  L^2(\Omega)^3 \,\, \text{ is given by } \,\, (T\f)(\x) =  k^2 P (\f +\v)
\end{align}
where again $\v \in H_{\mathrm{loc}}(\curl,\R^3)$ is the solution to \eqref{eq:secondOrder}--\eqref{eq:radiationCond} with $\Ei = \f$. Due to the well-posedness 
of the direct problem and the boundedness of $P$ it is clear that $T$ is a bounded linear operator on $L^2(\Omega)^3$, and hence $G = H^*T$ is also a compact operator. This gives that $\mathcal{F} = H^* T H$. 
This factorization  is used to derive the Factorization method for reconstruction $\Omega$ in \cite{Kirsc2004}. In order to continue we first make the following assumptions on the domain and coefficients. 
\begin{assumption}\label{assume}
We will assume that $\Omega$ is a Lipschitz bounded domain in $\R^3$. The contrast is assumed that $P \in L^{\infty}(\Omega)^{3 \times 3}$.
We lastly assume that the wave number $k$ is not a transmission eigenvalue i.e. the only solution to the homogeneous system in $L^2(\Omega)^3\times L^2(\Omega)^3$
\begin{align}
\curl\curl {\bf w}-k^2 \varepsilon_r {\bf w}=0 \,\, \text{ and } \,\, \curl\curl \boldsymbol{\varphi}-k^2  \boldsymbol{\varphi}=0\quad \text{ in} \,\ \Omega \label{te-problem1} \\
\nu \times {\bf w} = \nu \times \boldsymbol{\varphi}\,\, \text{ and } \,\, \nu \times \curl {\bf w} = \nu \times  \curl \boldsymbol{\varphi} \quad \text{on } \partial \Omega \label{te-problem2}
\end{align}
is trivial. 
% and for a non-absorbing material we assume $P =\text{Diag}(\boldsymbol{\alpha})$ such that $\boldsymbol{\alpha} \in L^{\infty}(\Omega)^3 \cap H^{1}_0 (\Omega)^{3}$. Where $H^{1}_0 (\Omega)^{3}$ is the standard Sobolev space of $L^{2}$ vector-valued functions with a weak gradient in $L^{2}$ and zero trace on $\partial \Omega$. 
\end{assumption}
%Note that extra regularity is also needed for the case of a non-absorbing material in \cite{kirsch2004} when analyzing the Factorization method but the regularity here is much weaker. The assumption that the contrast $P =\text{Diag}(\boldsymbol{\alpha})$ such that $\boldsymbol{\alpha} \in L^{\infty}(\Omega)^{3} \cap H^{1}_0 (\Omega)^{3}$ gives that it is bounded in $\Omega$ which is required for the forward scattering problem \eqref{eq:secondOrder}--\eqref{eq:radiationCond} and has an $L^{2}$ gradient which will be needed in our analysis. 
%
%We lastly assume that the wave number $k$ is not a transmission eigenvalue i.e. the only solution to the homogeneous system in $L^2(\Omega)^3\times L^2(\Omega)^3$
%\begin{align}
%\curl\curl {\bf w}-k^2 \varepsilon_r {\bf w}=0 \,\, \text{ and } \,\, \curl\curl \boldsymbol{\varphi}-k^2  \boldsymbol{\varphi}=0\quad \text{ in} \,\ \Omega \label{te-problem1} \\
%\nu \times {\bf w} = \nu \times \boldsymbol{\varphi}\,\, \text{ and } \,\, \nu \times \curl {\bf w} = \nu \times  \curl \boldsymbol{\varphi} \quad \text{on } \partial \Omega \label{te-problem2}
%\end{align}
%is trivial. 
It is known that the set of real transmission eigenvalues is at most discrete for a real-valued permittivity and is empty for a complex-valued permittivity. See \cite{Cakon2011} and the references therein for the analysis of transmission eigenvalue problems. The following results can be found in \cite{Kirsc2004}. These results will be critical in the later section where we study the orthogonality and direct sampling methods. 

\begin{theorem}\label{fm-results}
Let the operators $H$ and $T$ be as defined in \eqref{operatorH}--\eqref{operatorT} and let $\p,\, \y_s \in \R^3$.
\begin{enumerate}
\item[a.] The operator $H$ is compact and injective.
\item[b.]  The far field operator $\mathcal{F}$ is compact and injective. 
\item[c.] The far field operator has the factorization $\mathcal{F} = G H $ where $G=H^* T$. 
\item[d.] Let $\phi_{\y_s}(\di) = (\di \times\p)\times \di  \, e^{-ik\di \cdot \y_s}$ then $\phi_{\y_s} \in \mathrm{Range}(H^*) \iff \y_s \in \Omega$.
\end{enumerate}
\end{theorem}

The following coercivity result from  \cite{Hadda2014} is important for our theoretical analysis of the sampling methods.
\begin{theorem}\label{fm-results3}
Let the operator $T$ be as defined in \eqref{operatorT}. If the matrix-valued contrast $P  \in L^{\infty}(\Omega)^{3 \times 3} $ satisfies either 
\begin{enumerate}
\item The imaginary part $\Im(P)$ is uniformly positive definite or 
\item There is a constant $\alpha \geq 0$ such that $\Re(P)+\alpha \Im(P)$ is uniformly positive definite and $\Im(P)$ positive semidefinite. 
\end{enumerate}
Then provided that $k$ is not a transmission eigenvalue the operator $T$ is coercive on $\overline{\mathrm{Range}(H)}$.
\end{theorem}
\begin{proof}
See the proof of Theorem 10 in  \cite{Hadda2014} for details. 
\end{proof}

\section{Orthogonality sampling method}\label{OSM-section}

This section is dedicated to studying the OSM for solving the inverse problem. 
We first define the imaging functional for the OSM and  then prove a resolution 
analysis as well as stability for the imaging functional which is the main theoretical result of this section. After that
we  show an equivalence between the imaging functionals of the OSM 
and the DSM.

\textbf{The imaging functional.}  Let $\y_s$ be the sampling points in the imaging process and $\mathbf{p} \in \R^3$
is a fixed vector. We are interested in imaging of the scatterer $\Omega$ given the far field pattern $\u^\infty(\widehat{\x},\di,\q)$, for all $\widehat{\x},\di \in \mathbb{S}^2$ and polarization vector $\q=(\di \times\p)\times \di$.  We define the imaging functional  $\mathcal{I}_{OSM}$ for the OSM  as
\begin{align}
\label{indicator_OSM}
\mathcal{I}_{OSM}(\y_s) :=  \int_{\S^2} \left| \int_{\S^2} \u^\infty(\widehat{\x},\di)(\di \times \p) \times \di \cdot (\widehat{\x}\times\p)\times \widehat{\x}  \, e^{ik\widehat{\x} \cdot \y_s}  \d s(\widehat{\x}) \right|^2 \d s(\di).
\end{align}
The use of $(\di \times\p)\times \di$ in the imaging functional is useful for the analysis in this section and    to have the polarization vector belonging to $L^2_t(\S^2)$ for any choice of $\p \in \R^3$.  We now wish to write the imaging functional in terms of the far field operator. This will allow us to study the the imaging functional using the factorization analysis in the previous section.

\begin{lemma}
\label{OSM}
The imaging functional $\mathcal{I}_{OSM}$ satisfies 
\begin{align*}
\mathcal{I}_{OSM}(\y_s) = \|\p \cdot \mathcal{F}\phi_{\y_s}\|^2_{L^2(\S^2)},
\end{align*}
where  $\phi_{\y_s} \in L^2_t(\S^2)$ is given by
\begin{align*}
\phi_{\y_s}(\di) = (\di \times\p)\times \di  \, e^{-ik\di \cdot \y_s}.
\end{align*}
\end{lemma}
\begin{proof} 
We use the reciprocity relation and interchange the roles of $\widehat{\x}$ and $\di$
\begin{align*}
\mathcal{I}_{OSM} (\y_s)
& =  \int_{\S^2} \left| \int_{\S^2} \u^\infty(\widehat{\x},\di)(\di \times \p) \times \di \cdot (\widehat{\x}\times\p)\times \widehat{\x}  \, e^{ik\widehat{\x} \cdot \y_s}  \d s(\widehat{\x}) \right|^2 \d s(\di) \\
 &=  \int_{\S^2} \left|  \int_{\S^2} [\u^\infty(-\di,-\widehat{\x})]^\top (\di \times \p) \times \di  \cdot (\widehat{\x}\times\p)\times \widehat{\x}  \, e^{ik\widehat{\x} \cdot \y_s} \, \d s(\widehat{\x}) \right |^2 \d s(\di)   \\
 &=  \int_{\S^2} \left|  \int_{\S^2} [\u^\infty(\di,\widehat{\x})]^\top (\di \times \p) \times \di  \cdot (\widehat{\x}\times\p)\times \widehat{\x}  \, e^{-ik\widehat{\x} \cdot \y_s} \, \d s(\widehat{\x}) \right |^2 \d s(\di)   \\
& = \int_{\S^2} \left|  \int_{\S^2} \u^\infty(\di,\widehat{\x}) ((\widehat{\x}\times\p)\times \widehat{\x}  \, e^{- ik\widehat{\x} \cdot \y_s} ) \cdot (\di \times \p) \times \di \, \d s(\widehat{\x}) \right |^2 \d s(\di)   \\
& = \int_{\S^2} \left| (\widehat{\x}\times\p)\times \widehat{\x}\cdot \int_{\S^2} \u^\infty(\widehat{\x},\di) ({\di}\times\p)\times {\di}  \, e^{- ik{\di} \cdot \y_s}  \, \d s({\di}) \right |^2 \d s(\widehat{\x})   \\
 & =\left \|(\widehat{\x}\times\p)\times \widehat{\x} \cdot \mathcal{F}\phi_{\y_s} \right \|^2_{L^2(\S^2)}.
\end{align*}
Now using  the   identity 
$$
(\widehat{\x}\times\p)\times \widehat{\x} = \p ( \widehat{\x} \cdot  \widehat{\x} ) -  \widehat{\x} ( \widehat{\x} \cdot \p) = \p -  \widehat{\x} ( \widehat{\x} \cdot \p),
$$
and the fact that  $\mathcal{F} \phi_{\y_s}(\widehat{\x})\cdot \widehat{\x}  = 0$, proves the claim.
\end{proof}

We have shown that the imaging functional can be represented in terms of the far field operator. We now turn our attention to studying its properties. Just as in the previous section we will assume that the contrast satisfies the assumptions of Theorem \ref{fm-results3}. This is to insure the coercivity property of the middle operator $T$.  

\begin{theorem}\label{osm}
For every $\y_s \in \Omega$  
the imaging functional $\mathcal{I}_{OSM}(\y_s)$ is bounded from below by a  positive constant. 
 Moreover, for $\y_s \notin \Omega$ the imaging functional satisfies 
 $$
 \mathcal{I}_{OSM}(\y_s) = O\left( \frac{1}{\mathrm{dist}(\y_s,\Omega)^2} \right), \quad  \text{as } \mathrm{dist}(\y_s,\Omega) \to \infty.
 $$
\end{theorem}
\begin{proof}
We first observe that
 using again   $\mathcal{F} \phi_{\y_s}(\widehat{\x})\cdot \widehat{\x}  = 0$ and the identity
$
(\widehat{\x}\times\p)\times \widehat{\x} = \p -  \widehat{\x} ( \widehat{\x} \cdot \p)
$
we have 
\begin{align*}
 \langle \mathcal{F} \phi_{\y_s}, \phi_{\y_s}  \rangle_{L^2_t(\S^2)} 
 &=   \int_{\S^2}  \mathcal{F} \phi_{\y_s}(\widehat{\x}) \cdot \big( (\widehat{\x}\times\p)\times \widehat{\x} \, e^{ik\widehat{\x} \cdot \y_s} \big)\, \d s( \widehat{\x})  \\
 &= \int_{\S^2}  \mathcal{F} \phi_{\y_s}(\widehat{\x}) \cdot  \p \, e^{ik\widehat{\x} \cdot \y_s} \, \d s( \widehat{\x}).
\end{align*}
 Therefore, Lemma~\ref{OSM} and Cauchy-Schwarz inequality deduce that 
 \begin{align*}
\left|  \langle \mathcal{F} \phi_{\y_s}, \phi_{\y_s}  \rangle_{L^2_t(\S^2)}\right|^2 \leq |\S^2|^2 \| \p \cdot \mathcal{F} \phi_{\y_s} \|^2_{L^2(\S^2)} 
=|\S^2|^2 \mathcal{I}_{OSM}(\y_s).
 \end{align*}
Here $|\S^2|$ is the surface  area of  $\S^2$. From the factorization $\mathcal{F} = H^*TH$ and the coercive property of $T$  we have that
\begin{align*}
\mathcal{I}_{OSM}(\y_s)  &\geq \frac{1}{|\S^2|^2}\left| \langle H^*TH \phi_{\y_s}, \phi_{\y_s}  \rangle_{L^2_t(\S^2)}  \right |^2 \\
& = \frac{1}{|\S^2|^2} \left| \langle TH \phi_{\y_s}, H\phi_{\y_s}  \rangle_{L^2_t(\S^2)}  \right |^2 \\
& \geq C \left\| H \phi_{\y_s} \right\|^4,
\end{align*}
where $C$ is some positive constant. Now let $\y_s \in \Omega$ then we have that $\phi_{\y_s} \in \mathrm{Range}(H^*)$  by Theorem \ref{fm-results}. Therefore,
we have 
 $\phi_{\y_s} = H^*\varphi_{\y_s}$ for some $\varphi_{\y_s} \neq 0$, and
\begin{align*}
 \big\| H \phi_{\y_s} \big\| &=  \frac{\big\| H \phi_{\y_s} \big\| \|\varphi_{\y_s}\|}{\| \varphi_{\y_s}\|}  \\
 &\geq  \frac{\langle  H \phi_{\y_s},\varphi_{\y_s} \rangle}{\| \varphi_{\y_s}\|}  \\
 & =  \frac{\langle   \phi_{\y_s},H^*\varphi_{\y_s} \rangle}{\| \varphi_{\y_s}\|}   =  \frac{\|  \phi_{\y_s}\| ^2}{\| \varphi_{\y_s}\|} >0,
\end{align*}
proving the first statement of the theorem.

We now show that the imaging functional decays as $\text{dist}(\y_s,\Omega) \to \infty$. To do so, we first observe  from Lemma~\ref{OSM}  and the factorization $\mathcal{F} = GH$ that
$$
 \mathcal{I}_{OSM}(\y_s) = \|\p \cdot GH \phi_{\y_s}\|^2_{L^2(\S^2)} \leq |\p|^2\|G\|^2 \|H \phi_{\y_s}\|^2.
$$
Therefore, we can show that $\big\| H \phi_{\y_s} \big\|^2$ satisfies the decay property and use the upper bound on the imaging functional given above. Here we let $Y^{m}_{\ell}$ denote the spherical harmonics which form a complete orthonormal system on $L^2(\S^2)$. Now, recall the Funk-Hecke formula (see for e.g. \cite{Colto2013}) 
$$ \int_{\S^2} Y^{m}_{\ell} (\di) e^{-ik\di \cdot \x} \, \d s(\di) = \frac{4\pi}{i^{\ell}} Y^{m}_{\ell} (\hat{\x}) j_\ell (k|\x|) \quad \text{ for } \quad m \in \N \cup\{0\} \,\, \text{ and } \,\, \ell = -m, \dots , m$$
where $j_\ell$ is the first kind spherical Bessel function of order $\ell$. In particular,
$$Y^0_0 = \frac{1}{\sqrt{4\pi}}, \quad j_0 (t) =  \frac{\sin(t)}{t}.$$
Just as in recent works \cite{Harri2019, Nguye2019} we will use the Funk-Hecke formula to show that $\big\| H \phi_{\y_s} \big\|^2$ decays as $\text{dist}(\y_s,\Omega) \to \infty$. 
Indeed, using the Funk-Hecke formula for $m = \ell = 0$ with the formula $\curl_\z\curl_\z (\p e^{-ik\di \cdot \z}) = - k^2 (\di \times\p)\times \di  e^{-ik\di \cdot \z}$  straightforward calculations gives
\begin{align*}
  \int_{\S^2} (\di \times\p)\times \di e^{-ik\di \cdot \z} \, \d s(\di) 
 & = -\frac{1}{k^2}\curl_\z\curl_\z \int_{\S^2}  \p e^{-ik\di \cdot \z} \, \d s(\di) \\
 &= -\frac{4\pi}{k^2}\curl_\z\curl_\z(\p j_0(k|\z|))
 = -\frac{4\pi}{k^2} \left(
\begin{array}{ccc}
v_1(\z,\mathbf{p}) \\
v_2 (\z,\mathbf{p})\\
v_3(\z,\mathbf{p})
\end{array}
\right)
\end{align*}
where, for $j = 1,2,3$, 
\begin{align*}
v_j(\z,\mathbf{p}) =  k^2\left(  \frac{ |\z|^2-(\mathbf{p}\cdot \z)z_j}{|\z|^2} \right) j_0(k|\z|) - 3\left(  \frac{(\mathbf{p}\cdot \z)z_j}{|\z|^4} \right)(\cos(k|\z|) -  j_0(k|\z|) ).
\end{align*}
It is obvious that $v_j = O(1/|\z|)$ as $|\z| \to \infty$. Therefore
\begin{align*}
(H \phi_{\y_s} ) (\x)  =  \int_{\S^2} (\di \times\p)\times \di e^{-ik\di \cdot (\y_s -\x)} \, \d s(\di) = -\frac{4\pi}{k^2} \left(
\begin{array}{ccc}
v_1(\y_s -\x,\mathbf{p}) \\
v_2 (\y_s -\x,\mathbf{p})\\
v_3(\y_s -\x,\mathbf{p})
\end{array}
\right)
\end{align*}
which leads to
$$
\| H \phi_{\y_s} \|^2 = O\left(\frac{1}{\text{dist}(\y_s,\Omega)^2} \right), \quad
\text{dist}(\y_s,\Omega) \to \infty
$$
proving the claim.
\end{proof}

This resolution analysis implies the for any sampling points $\y_s \in \Omega$ the imaging functional is strictly positive and will decay as $\y_s$ moves away from the scatterer $\Omega$. Therefore, one can plot $\mathcal{I}_{OSM}(\y_s)$ in order to recover the scatterer.  Using the imaging functional $\mathcal{I}_{OSM}(\y_s)$ has the advantage that one does not have to solve an ill-posed equation at each sampling point. Other sampling methods such as the Linear sampling method \cite{Cakon2011} and  Factorization method \cite{Kirsc2008} requires one to solve an ill-posed equation involving the far field operator at each sampling point. This requires one to compute the singular-value decomposition of the far field operator as well as applying a regularization scheme. Also stability of these methods
with respect to noise added to the data is not justified where as the orthogonality sampling method only requires one to compute an inner-product and by the following theorem we see that the  imaging functional is stable with respect to noise added to the far field data.

In many physical applications one only knows the far field data up to some `small' perturbation. Now, assume that we only know the far field operator up to a perturbation, that means  we  have access to its noisy version $\mathcal{F}_{\delta}$, 
which satisfies
$$
\| \mathcal{F} - \mathcal{F}_{\delta}\| \leq \delta \| \mathcal{F} \|,
$$
for some $\delta >0$. Here $\mathcal{F}_{\delta}$ represents the measured far field operator physical experiments. We now give a stablity estimate for the imaging functional $\mathcal{I}_{OSM}(\y_s)$.

\begin{theorem}[stability estimate]
\label{stabilityOSM} 
%For all $\widehat{\x},\di \in \mathbb{S}^2$,  let $\mathcal{D}(\widehat{\x},\di) := \u^\infty(\widehat{\x},\di)(\di\times\p)\times \di$ be the far field data 
%in the imaging functional $\mathcal{I}_{OSM}$ in~\eqref{indicator_OSM}. 
Denote   by $\mathcal{I}_{OSM,\delta}$ the imaging functional corresponding to noisy far field operator  $\mathcal{F}_{\delta}$. Then
\begin{align*}
\mathcal{I}_{OSM}(\y_s) - \mathcal{I}_{OSM,\delta}(\y_s) \leq  |\p|^2|\S|^2 \|\mathcal{F} \|^2 (\delta^2 + 2\delta), 
\quad \text{for all } \y_s \in \R^3,
\end{align*}
where $|\S^2|$ is again the surface area of $\S^2$, and $\p \in \R^3$ is the polarization vector in the definition of the imaging functional~\eqref{indicator_OSM}. 
\end{theorem}
\begin{proof} 
Using Lemma~\ref{OSM}, the triangle inequality and  the Cauchy-Schwarz inequality we have
\begin{align*}
&\mathcal{I}_{OSM}(\y_s) - \mathcal{I}_{OSM,\delta}(\y_s) =  \|\p \cdot \mathcal{F}\phi_{\y_s}\|^2_{L^2(\S^2)} - \|\p \cdot \mathcal{F}_\delta\phi_{\y_s}\|^2_{L^2(\S^2)} \\ 
&\leq \|\p \cdot (\mathcal{F}\phi_{\y_s} -\mathcal{F}_\delta\phi_{\y_s}) \|_{L^2(\S^2)} \left( \|\p \cdot \mathcal{F}\phi_{\y_s} \|_{L^2(\S^2)}+ \|\p \cdot \mathcal{F}_\delta\phi_{\y_s} \|_{L^2(\S^2)} \right)\\
&\leq|\p| |\S^2| \| \mathcal{F} -\mathcal{F}_\delta \| \left( 2\|\p \cdot \mathcal{F}\phi_{\y_s} \|_{L^2(\S^2)}  +  \|\p \cdot (\mathcal{F}\phi_{\y_s} -\mathcal{F}_\delta\phi_{\y_s}) \|_{L^2(\S^2)}  \right)\\
&\leq|\p| |\S^2|  \| \mathcal{F}\| \delta \left( 2|\p| |\S^2| \|\mathcal{F}\|  +|\p| |\S^2|  \| \mathcal{F}\| \delta \right)\\
&\leq |\p|^2|\S|^2 \|\mathcal{F} \|^2 (\delta^2 + 2\delta),
\end{align*}
proving the theorem.
\end{proof}

%\begin{theorem}[stability estimate]
%\label{stabilityOSM} 
%Denote   by $\mathcal{I}_{OSM,\delta}$ the imaging functional where data $\u^\infty(\widehat{\x},\di)\p $ is replaced by 
%its noisy version  $\u_\delta^\infty(\widehat{\x},\di)\p $.  Then
%\begin{align*}
%\mathcal{I}_{OSM} (\y_s) - \mathcal{I}_{OSM,\delta}(\y_s) \leq C \| \u^\infty(\widehat{\x},\di)\p   - \u_\delta^\infty(\widehat{\x},\di)\p  \|_{L^2(\mathbb{S}^2\times \mathbb{S}^2)^3}, 
%\quad \text{for all } \y_s \in \R^3,
%\end{align*}
%where $C$ is some positive constant.
%\end{theorem}

%%%%%%%%%%%%%%%%%%%%%%%%%%%%%%%%%%%%%%%%%%%%%%%%%
%\section{Direct and generalized linear sampling methods}\label{DSM-section}
%This section is dedicated to studying the direct and generalized linear sampling methods for the inverse scattering problem. 
Motivated by the work in \cite{Liu2017} we prove that the imaging functionals for the OSM and DSM are equivalent. The DSM was considered in \cite{Alzaa2017} for the far field operator corresponding to the magnetic field. In their analysis they assume that the  contrast is scalar-valued and use the factorization established in \cite{Kirsc2008}. The analysis presented here is valid for a matrix-valued contrast. Moreover, we give an explicit decay rate for the imaging function. 
%The generalized linear sampling method which was recently introduced in~\cite{GLSM} is also discussed for the case of negative value contrast $P$.
To this end, we again let the sampling points $\y_s$ and fixed vector $\mathbf{p} \in \R^3$, the imaging functional  for the DSM is defined as
\begin{align}
\label{indicator}
\mathcal{I}_{DSM}(\y_s) := \left| \int_{\S^2} e^{-ik\di \cdot \y_s} \int_{\S^2} \u^\infty(\widehat{\x},\di)(\di\times\p)\times \di \cdot \ol{\phi_{\y_s}(\widehat{\x})} \d s(\widehat{\x}) \d s(\di)  \right|,
\end{align}
where 
$$\phi_{\y_s}(\widehat{\x}) = (\widehat{\x} \times\p)\times \widehat{\x} \, e^{-ik\widehat{\x}\cdot \y_s}.$$
Now, similarly to Theorem~\ref{osm} we will represent the functional in terms of the far field operator. This will be used to prove the equivalence with OSM. 

\begin{lemma}\label{imaging-func}
The imaging  functional $\mathcal{I}_{DSM}$ for the DSM satisfies 
$$
\mathcal{I}_{DSM}(\y_s)  = 
\left| \langle \mathcal{F} \phi_{\y_s}, \phi_{\y_s}  \rangle_{L^2_t(\S^2)}  \right |.
$$
\end{lemma}
\begin{proof} 
We use the reciprocity relation and interchange the roles of $\widehat{\x}$ and $\di$
\begin{align*}
\mathcal{I}_{DSM}(\y_s) 
& =  \left| \int_{\S^2} e^{-ik\di \cdot \y_s} \int_{\S^2} [\u^\infty(-\di,-\widehat{\x})]^\top(\di\times\p)\times \di \cdot (\widehat{\x}\times\p)\times \widehat{\x}  \, e^{ik\widehat{\x} \cdot \y_s}  \d s(\widehat{\x}) \d s(\di)  \right| \\
 &= \left| \int_{\S^2}  \int_{\S^2} \u^\infty(\di,\widehat{\x}) (\widehat{\x}\times\p)\times \widehat{\x}  \, e^{-ik\widehat{\x} \cdot \y_s}  \cdot  (\di\times\p)\times \di \, e^{ik\di \cdot \y_s}\, \d s(\widehat{\x}) \d s(\di)  \right| \\
 &=  \left| \int_{\S^2}  \int_{\S^2} \u^\infty(\widehat{\x},\di) (\di\times\p)\times \di  \, e^{-ik \di \cdot \y_s}  \d s(\di) \cdot (\widehat{\x}\times\p)\times \widehat{\x} \, e^{ik\widehat{\x} \cdot \y_s} \d s( \widehat{\x})  \right| \\
 & =\left| \langle \mathcal{F} \phi_{\y_s}, \phi_{\y_s}  \rangle_{L^2_t(\S^2)}  \right |
\end{align*}
proving the claim.
\end{proof}

From Theorem \ref{osm} and \ref{imaging-func} we have all we need to derive the equivalence of the two imaging functionals studied in this section. Again we assume that the contrast satisfies the assumptions of Theorem \ref{fm-results3}. The equivalence of the two sampling methods is proven in the following result.

\begin{theorem}
\label{equivalent}
There
exists positive constants $c_1$ and $c_2$ such that 
\begin{align*}
c_1\mathcal{I}_{OSM} \leq \mathcal{I}_{DSM} \leq c_2\sqrt{\mathcal{I}_{OSM}}.
\end{align*}
\end{theorem}
\begin{proof}
From Lemma~\ref{imaging-func}, the factorization $\mathcal{F} = H^*TH$ and the coercive property of $T$ we obtain 
$$
 \mathcal{I}_{DSM}(\y_s)  \geq C \big\| H \phi_{\y_s} \big\|^2 
$$
where $C$ is some positive constant. The first inequality of the theorem hence follows from the estimate
$$
 \mathcal{I}_{OSM}(\y_s) = \|\p \cdot GH \phi_{\y_s}\|^2_{L^2(\S^2)} \leq |\p|^2\|G\|^2 \|H \phi_{\y_s}\|^2.
$$
The following estimate is from the beginning of the proof of Theorem~\ref{osm}
 \begin{align*}
\left|  \langle \mathcal{F} \phi_{\y_s}, \phi_{\y_s}  \rangle_{L^2_t(\S^2)}\right|^2 \leq |\S^2|^2 \| \p \cdot \mathcal{F} \phi_{\y_s} \|^2_{L^2(\S^2)} 
=|\S^2|^2 \mathcal{I}_{OSM}(\y_s).
 \end{align*}
This estimate and  Lemma~\ref{imaging-func} implies the second inequality of the theorem.
\end{proof}

Following the method for proving Theorem~\ref{osm} we can obtain similar resolution analysis for the direct sampling method.
\begin{theorem}\label{dsm}
For every $\y_s \in \Omega$  
the imaging functional $\mathcal{I}_{DSM}(\y_s)$ is bounded from below by a  positive constant. 
 Moreover, for $\y_s \notin \Omega$ the imaging functional satisfies 
 $$
 \mathcal{I}_{DSM}(\y_s) = O\left( \frac{1}{\mathrm{dist}(\y_s,\Omega)^2} \right), \quad  \text{as } \mathrm{dist}(\y_s,\Omega) \to \infty.
 $$
\end{theorem}
We also have the stability of the DSM that can be proved using the triangle  and the Cauchy-Schwarz inequalities. The analysis in this section allows one to solve the 
inverse shape problem for electromagnetic scattering by plotting either $\mathcal{I}_{DSM}$ or $\mathcal{I}_{OSM}$. This amounts to a fast yet stable reconstruction method that only relies on the knowledge of the far-field data. Unlike some traditional sampling methods one does not need to minimize a non-linear functional at each sampling point.

\section{Numerical examples}
\label{imaging}
We present in this section several numerical examples to validate the performance of the 
OSM as well as the DSM. The simulations were carried on a Quad Core 3.6 GHz machine with 
 32GB RAM and the implementation was done using the computing software Matlab. 
 The synthetic data are generated by numerically solving the direct problem with 
 the spectral solver studied in~\cite{Nguye2019}. We solve the direct problem~\eqref{eq:secondOrder}--\eqref{eq:radiationCond}
 with incident field 
 $$
 \Ei(\x,\di_j, \q)=ik (\di_j\times \p) \times  \di_j e^{ik\x\cdot\di_j}, \quad  j = 1,2,\dots,N_\di 
 $$
 where   $\p = (1/\sqrt{3},-1/\sqrt{3},1/\sqrt{3})^\top$ 
 and $N_\di$ is the number of directions $\di_j$ that is specified below for each numerical example.  
 Likewise, we denote by $N_{\widehat{\x}}$ the number of points ${\widehat{\x}} \in \S^2$  where
the  far field pattern data are collected. The points   $\widehat{\x}$ and $\di$ that are chosen 
for generating the scattering data are almost uniformly distributed on $\S^2$.
Let 
 $$\mathcal{D}(\widehat{\x},\di) = \u^\infty(\widehat{\x},\di)(\di\times\p)\times \di$$ 
 be our synthetic far field data
 that has three components $\mathcal{D}_n, n = 1,2,3,$ and $\mathcal{D}_n$ 
 can be considered as an $N_{\widehat{\x}}\times N_\di$  matrix. 
To consider noisy data, we add artificial noise to  our synthetic data. 
More precisely, a complex-valued noise matrix $\mathcal{N}$ 
containing random numbers that are uniformly distributed in the complex square 
$$\{ a+\i b\, : \, |a| \leq 1, \, |b| \leq 1 \} \subset \C$$
 is added to the data matrix 
$\mathcal{D}_n$. Denoting by $\delta$ the noise level, the
noisy data matrix $\mathcal{D}_{n,\delta}$ is then given by
\[
 \mathcal{D}_{n,\delta}
 := \mathcal{D}_{n}
   + \delta\frac{\mathcal{N}}{\|\mathcal{N}\|_2} \left\|\mathcal{D}_{n} \right\|_2, \quad n = 1,2,3,
\]
where $\|\cdot\|_2$ is the matrix 2-norm. 
To define the anisotropic contrasts for the numerical examples, we
need the following diagonal matrix 
\begin{align}
\label{matrix}
A =
\left[\begin{array}{ccccc}
 1  &  0 &    0\\  
0 &   1.5 &   0\\ 
  0 &  0 &   1.2
\end{array}
\right].
\end{align}

\subsection{OSM and DSM  (Figure~\ref{fi1}).} 
We present in Figure~\ref{fi1} the reconstruction results of the OSM and DSM 
for an anisotropic scatterer including three different balls.  The scatterer is characterized 
by smoothly varying contrast $P(\x)$ defined by
\begin{align}
\Omega_1 &= \left\{\x \in \R^3: |\x - \mathbf{a}|^2 < 0.3^2, \mathbf{a} = (0.4,0,-0.45)^\top \right \} \nonumber \\
\Omega_2 &= \left\{\x \in \R^3: |\x - \mathbf{b}|^2 < 0.35^2, \mathbf{b} = (-0.4,0,0)^\top \right \} \nonumber\\
\Omega_3 &= \left\{\x \in \R^3: |\x - \mathbf{c}|^2 < 0.4^2, \mathbf{c}= (0.4,0,0.4)^\top \right \}\nonumber \\
P(\x) &= \begin{cases} 
A\exp\left(1 - \frac{0.3^2}{0.3^2 - |\x - \mathbf{a}|^2} \right), \quad & \x \in \Omega_1 \\
A\exp\left(1 - \frac{0.35^2}{0.35^2 - |\x - \mathbf{b}|^2} \right), \quad & \x \in \Omega_2 \\
A\exp\left(1 - \frac{0.4^2}{0.4^2 - |\x - \mathbf{c}|^2} \right), \quad & \x \in \Omega_3 \\
  0, \quad &  \text{else},
    \end{cases}
    \label{contrast}
\end{align}
where $A$ is given by~\eqref{matrix}. 

Here we choose wave number $k = 12$ (the corresponding wavelength is about 0.52)  
and $N_{\widehat{\x}}\times N_\di = 325\times 325$.  The far field data are perturbed by 30$\%$ noise.
In Figure~\ref{fi1} we also present 3D-visualizations of the exact geometry and its reconstructions using isosurface in Matlab.
The isovalue for  the isosurface plotting is chosen to be $1/3$ of the maximal value of the computed imaging functionals
 $\mathcal{I}_{OSM}$ and $\mathcal{I}_{DSM}$. It can be seen in Figure~\ref{fi1}  that both the OSM and the DSM are robust with noise added
 to the data and provide reasonable reconstructions
 of the scatterer. We also observe   that the reconstruction from $\mathcal{I}_{OSM}$ is more similar to that from $(\mathcal{I}_{DSM})^2$ 
 rather than the one from  $\mathcal{I}_{DSM}$. This  might reflect the estimate $\mathcal{I}_{DSM} \leq c_2 \sqrt{\mathcal{I}_{OSM}}$
 from the equivalent relation in Theorem~\ref{equivalent}.

\begin{figure}[ht!]
\centering
\subfloat[]{\includegraphics[width=3.9cm]{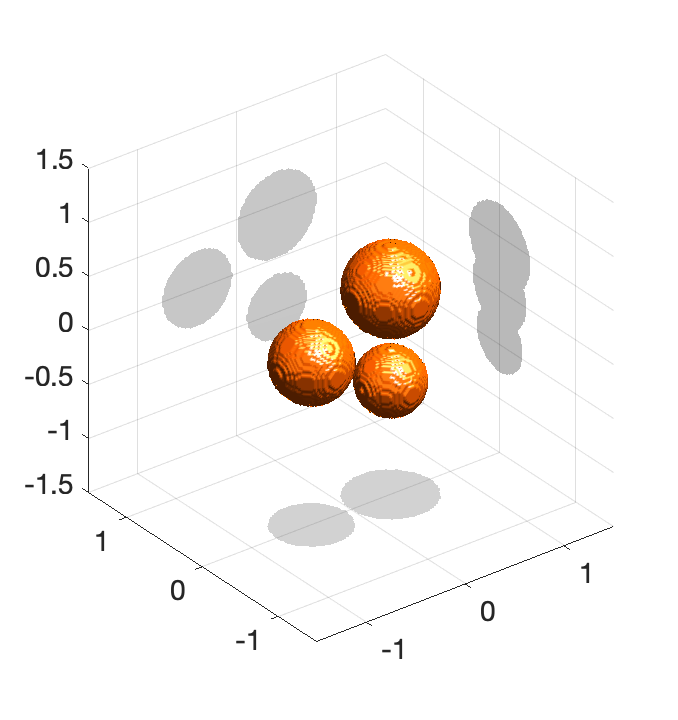}} \hspace{-0.25cm}
\subfloat[]{\includegraphics[width=3.9cm]{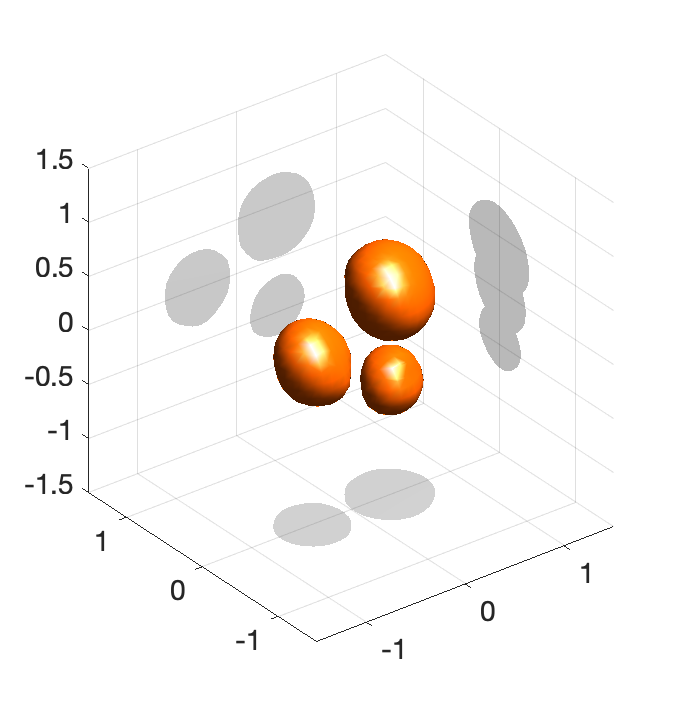}}  \hspace{-0.25cm}
\subfloat[]{\includegraphics[width=3.9cm]{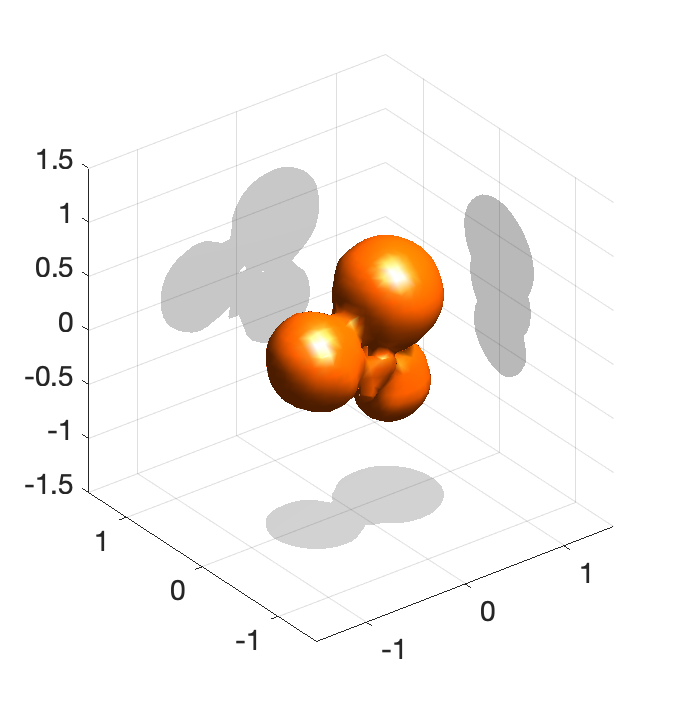}}  \hspace{-0.25cm}
\subfloat[]{\includegraphics[width=3.9cm]{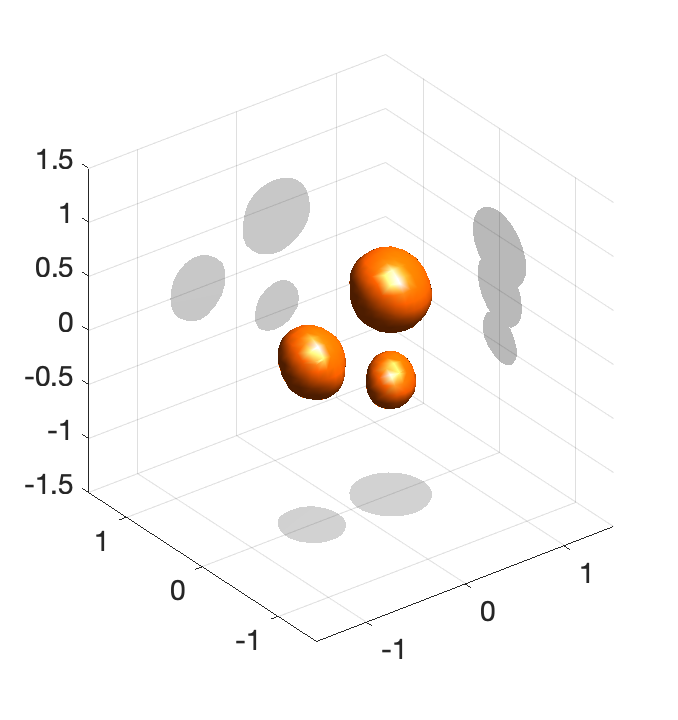}} \\
\subfloat[]{\includegraphics[width=3.8cm]{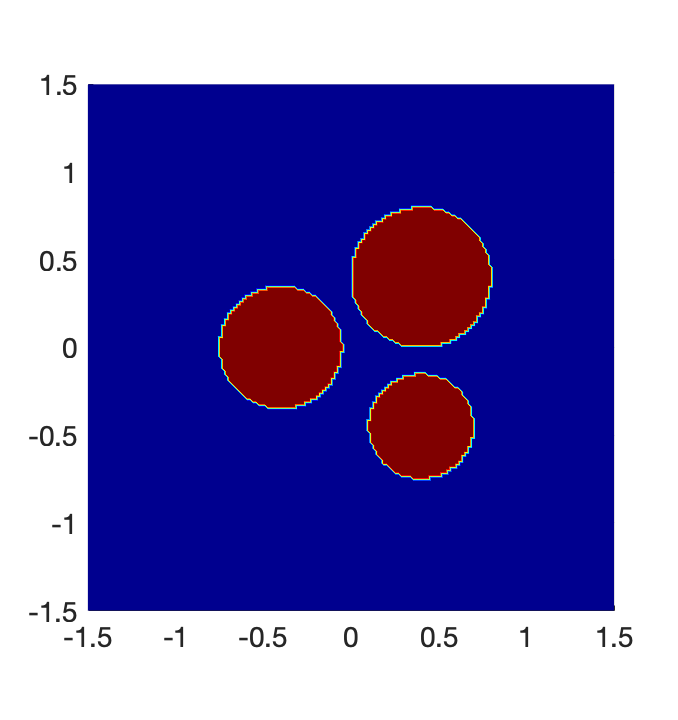}} \hspace{-0.15cm}
\subfloat[]{\includegraphics[width=3.8cm]{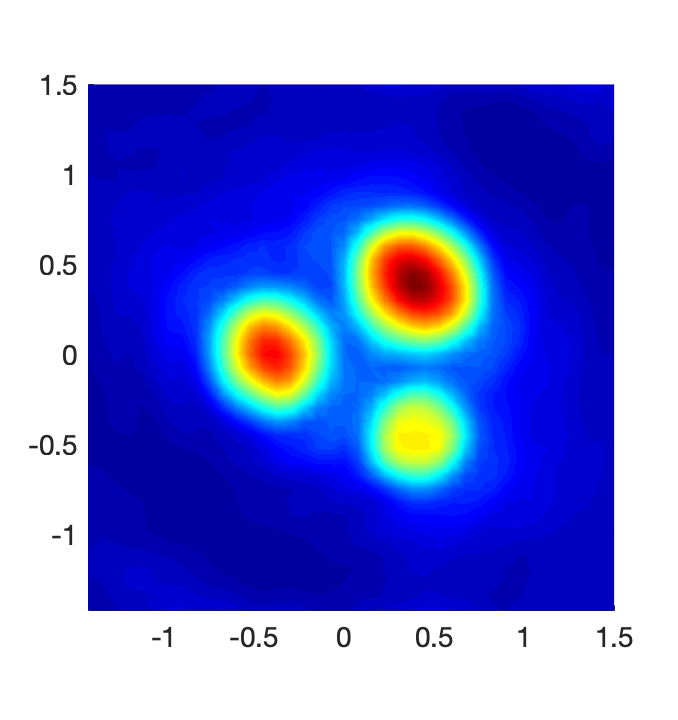}}  \hspace{-0.15cm}
\subfloat[]{\includegraphics[width=3.8cm]{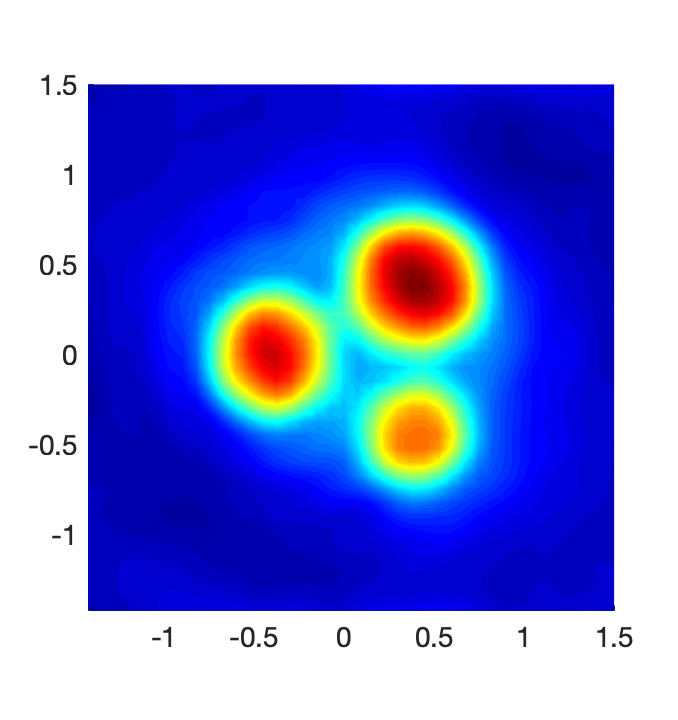}}  \hspace{-0.15cm}
\subfloat[]{\includegraphics[width=3.8cm]{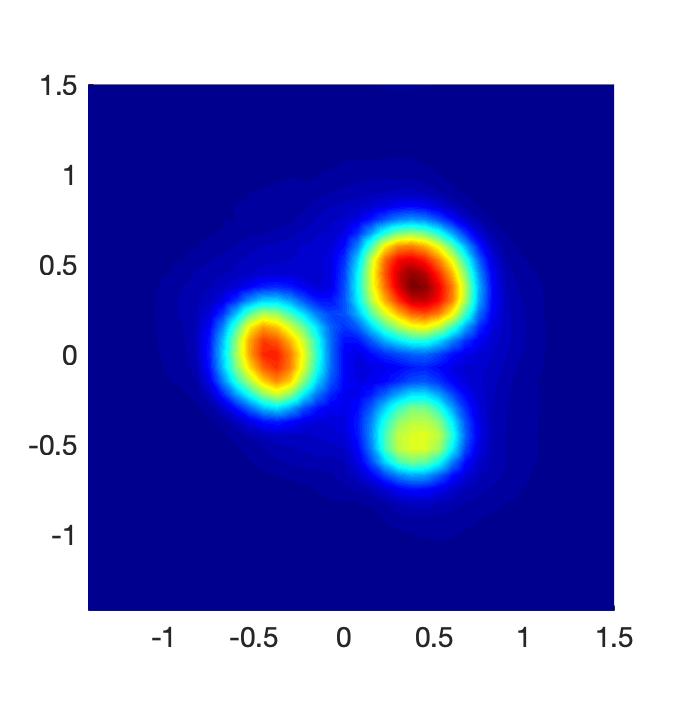}}  
\caption{Reconstruction  with OSM and DSM for the scatterer including three different balls. 
 $N_{\widehat{\x}}\times N_\di = 325\times 325$, wave number $k=12$. There is  30$\%$ noise added to the data.
 (a) Exact geometry. 
(b) Reconstruction with $\mathcal{I}_{OSM}$. (c) Reconstruction with $\mathcal{I}_{DSM}$. 
(d) Reconstruction with $(\mathcal{I}_{DSM})^2$.
(e) Cross-sectional view of the exact geometry. (f) Cross-sectional view of  the computed $\mathcal{I}_{OSM}$. 
(g) Cross-sectional view of the computed $\mathcal{I}_{DSM}$. (h) Cross-sectional view of the computed $(\mathcal{I}_{DSM})^2$.
The isovalue in  the isosurface plotting is $1/3$ of the maximal value of the computed imaging functionals.} 
\label{fi1}
\end{figure}

\subsection{OSM for highly noisy data  (Figure~\ref{fi2}).} 
The second example is presented in Figure~\ref{fi2} where we focus on 
performance of the sampling  method on  data perturbed by high 
amounts of noise.

Here consider the scatterer as in the previous example~\eqref{contrast},  $k = 12$ 
and $N_{\widehat{\x}}\times N_\di = 325\times 325$. 
From Figure~\ref{fi2} we can see that,  even there are high  levels of noise $\delta = 0.6$ and $0.9$
in the far field data,  that the OSM is still able to provide reasonable reconstructions. 
The computed images are not very different for 60$\%$  and 90$\%$ noise in the data.  
The solid performance of the method on noisy data can be justified by 
the stability of the method that is discussed in Theorem~\ref{stabilityOSM}.

\begin{figure}[ht!!]
\centering
\subfloat[]{\includegraphics[width=5cm]{3exp_Exact.png}} \hspace{0cm}
\subfloat[]{\includegraphics[width=5cm]{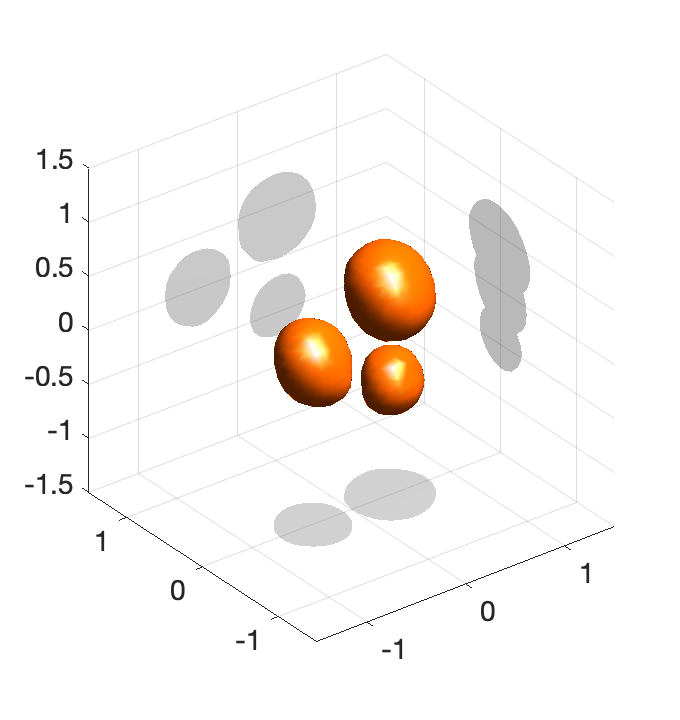}}  \hspace{0cm}
\subfloat[]{\includegraphics[width=5cm]{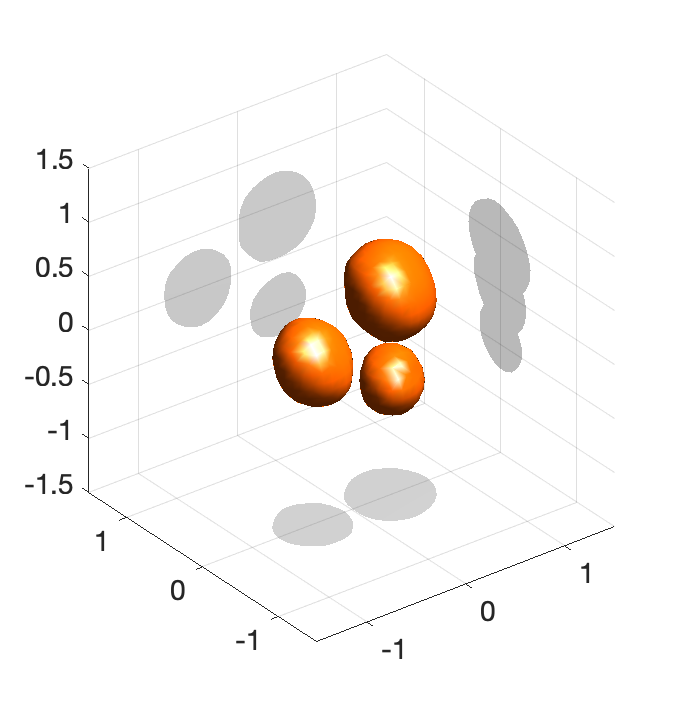}} \\
\subfloat[]{\includegraphics[width=4.5cm]{3exp_exact2D.png}} \hspace{0.3cm}
\subfloat[]{\includegraphics[width=4.5cm]{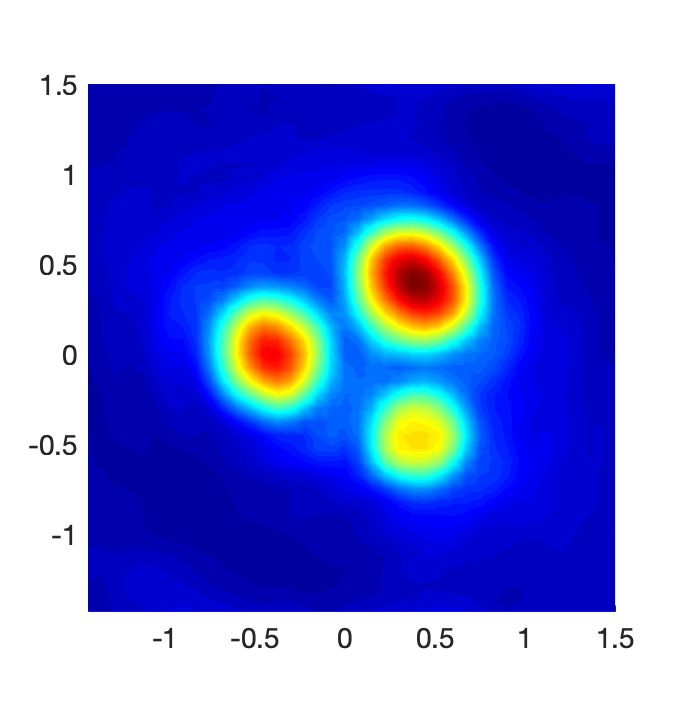}}  \hspace{0.3cm}
\subfloat[]{\includegraphics[width=4.5cm]{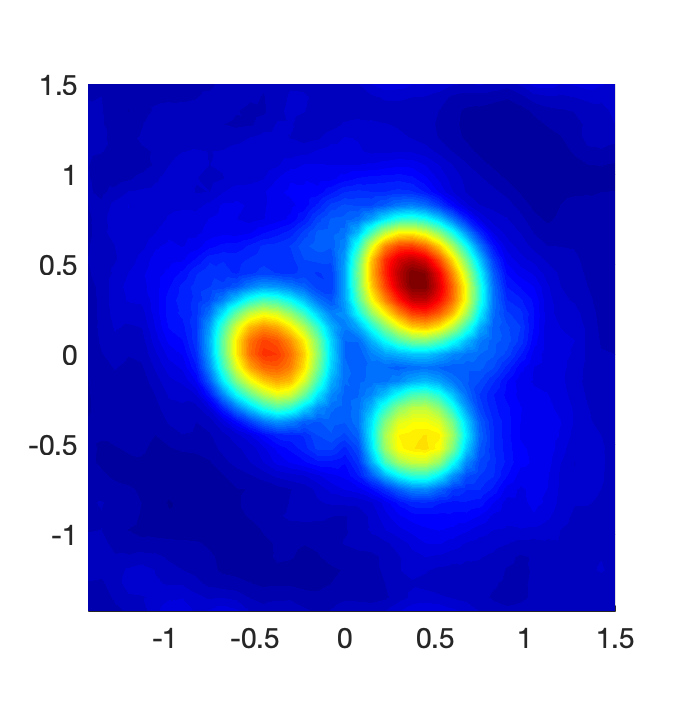}}  
\caption{Reconstruction  with highly noisy data using OSM. 
 $N_{\widehat{\x}}\times N_\di = 325\times 325$, wave number $k=12$. (a) Exact geometry. 
(b) Reconstruction for 60$\%$ noise. (c) Reconstruction for 90$\%$ noise. 
(d) Cross-sectional view of the exact geometry. (e) Cross-sectional view of the computed  
 $\mathcal{I}_{OSM}$ for 60$\%$ noise. 
(f) Cross-sectional view of the computed $\mathcal{I}_{OSM}$ for 90$\%$ noise.
The isovalue in  the isosurface plotting is $1/3$ of the maximal value of the computed $\mathcal{I}_{OSM}$.} 
\label{fi2}
\end{figure}

\subsection{OSM with less data (Figure~\ref{fi3}).} 
This is the focus of the third example that is to examine the performance of the OSM
on a smaller wave number and a smaller set of scattering data. We consider the same scatterer
as in the first and second examples~\eqref{contrast}. The data in this example are perturbed by 30$\%$ noise.
 We can see in Figure~\ref{fi3}(b) that for
 $N_{\widehat{\x}}\times N_\di = 325\times 325$ and wave number
$k = 6$ (wavelength is about 1.04) the reconstruction  is not as good as that of the case $k = 12$
(Figure~\ref{fi3}(c)). Although we can see two components of the scatterers in the reconstruction
  the shape and locations are not very accurate. 
The result is even worse if we have less data, see Figure~\ref{fi3}(d). More precisely, for  
$N_{\widehat{\x}}\times N_\di = 91 \times 91$ and $k = 12$, the reconstruction 
is no longer reasonable.

\begin{figure}[htt]
\centering
\subfloat[]{\includegraphics[width=4.5cm]{3exp_exact2D.png}} \hspace{0.3cm}
\subfloat[]{\includegraphics[width=4.5cm]{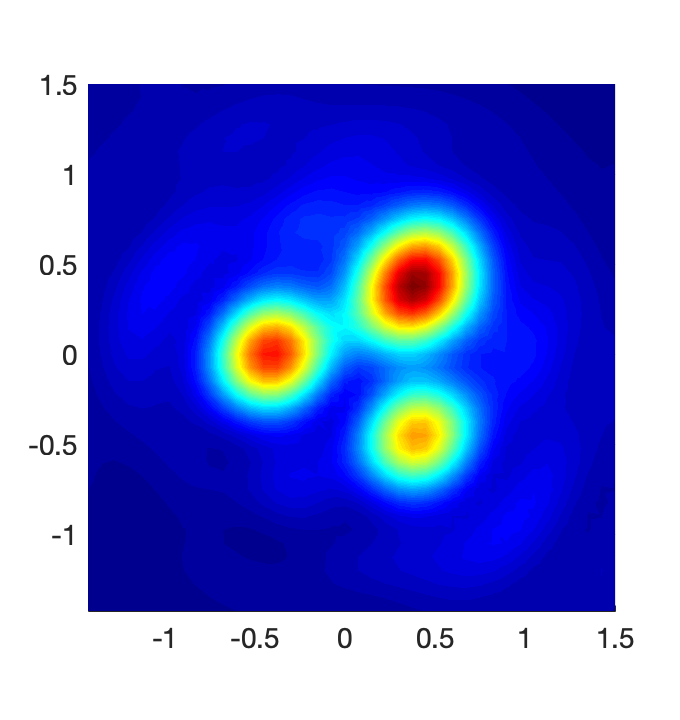}}  \hspace{0.3cm} \\
\subfloat[]{\includegraphics[width=4.5cm]{3exp_comput2D_OSM.png}}  \hspace{0.3cm}
\subfloat[]{\includegraphics[width=4.5cm]{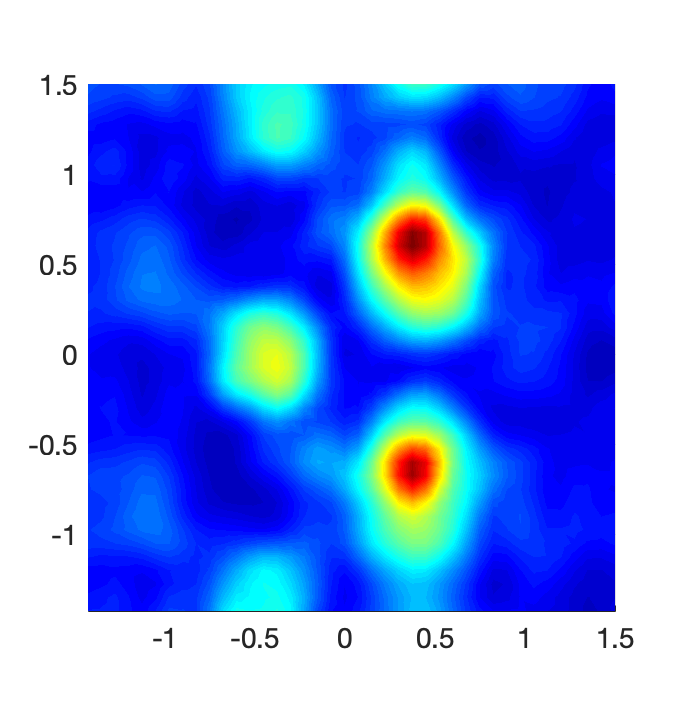}}  
\caption{Reconstruction  with a smaller wave number  and  a smaller amount of scattering data
in cross-sectional views of the computed  
 $\mathcal{I}_{OSM}$.  
There is  30$\%$ noise added to the data.
(a) Exact geometry. (b) 
 $N_{\widehat{\x}}\times N_\di = 325\times 325$ and $k = 6$. (c)
$N_{\widehat{\x}}\times N_\di = 325\times 325$ and  $k =12$. (d) 
 $N_{\widehat{\x}}\times N_\di = 91\times 91$ and  $k =12$.
 } 
 \label{fi3}
\end{figure}

\subsection{OSM for other types of scatterers (Figure~\ref{fi4}).} 
In this example we present in Figure~\ref{fi4} the  reconstruction results for  three different types 
of scatterers. We consider scatterers with  more complicated shapes and non-smooth  geometries.  
For the scatterer in Figure~\ref{fi4}(a) the anisotropic contrast $P(\x)$ is again a smoothly varying function defined by
\begin{align*}
D_1 &= \left\{\x \in \R^3: |\x - \mathbf{a}|^2 < 0.3^2, \mathbf{a} = (0,0,0.4)^\top \right \} \\
D_2 &= \left\{\x \in \R^3: |\x - \mathbf{b}|^2 < 0.5^2, \mathbf{b} = (0,0,-0.3)^\top \right \} \\
P(\x) &= \begin{cases} 
\frac{A}{2}\exp\left(1 - \frac{0.3^2}{0.3^2 - |\x - \mathbf{a}|^2} \right), \quad & \x \in D_1 \\
\frac{A}{2}\exp\left(1 - \frac{0.5^2}{0.5^2 - |\x - \mathbf{b}|^2} \right), \quad & \x \in D_2 \\
  0, \quad &  \text{else},
    \end{cases}
\end{align*}
where matrix $A$ is given by~\eqref{matrix}.
The contrast $P(\x)$ for the scatterers in Figures~\ref{fi4}(d) and~\ref{fi4}(g) are respectively equal to $A/2$  
and $A/4$ in $\Omega$ and 0 outside of $\Omega$.  Again, $N_{\widehat{\x}}\times N_\di = 325\times 325$, $k = 12$
and the data  are  perturbed by 30$\%$ noise. 
The pictures show that with the right amount of data the sampling method
is able to provide  good  reconstruction results for scatterers 
with more complicated shapes. We also observe that
for scatterers with non-smooth geometries like in  Figures~\ref{fi4}(d)  and~\ref{fi4}(g),
the isovalue in  the isosurface plotting should be  1/2 of the maximal value of the 
computed imaging functional $\mathcal{I}_{OSM}$ to give a better three-dimensional image.

\begin{figure}[ht!]
\centering
\subfloat[]{\includegraphics[width=5cm]{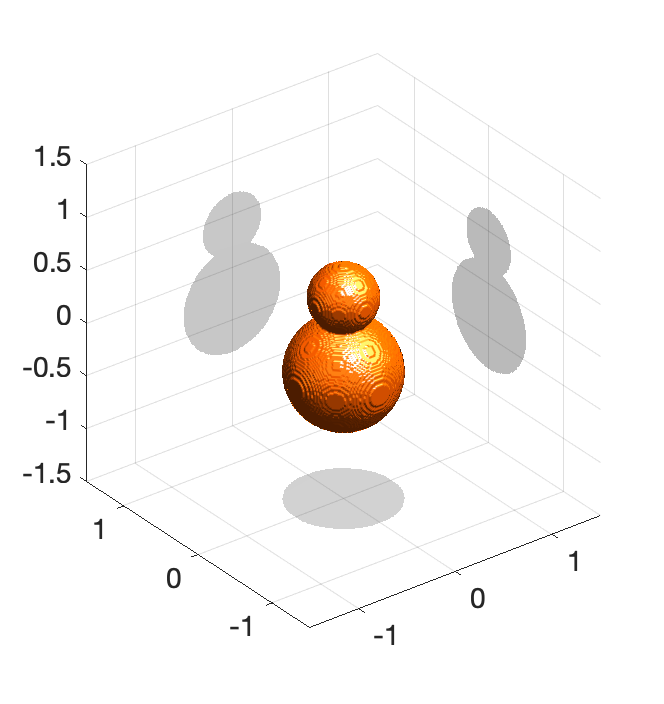}} \hspace{0.cm}
\subfloat[]{\includegraphics[width=5cm]{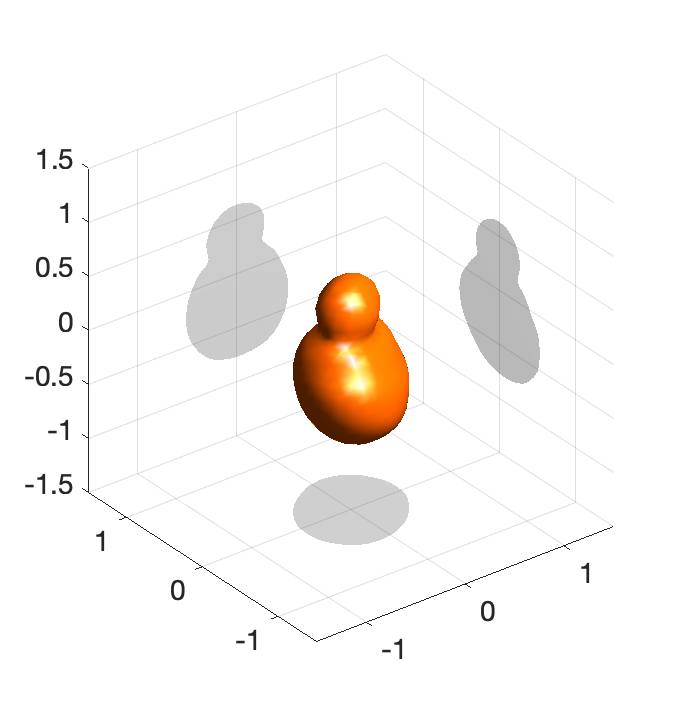}}  \hspace{0.cm}
\subfloat[]{\includegraphics[width=4.5cm]{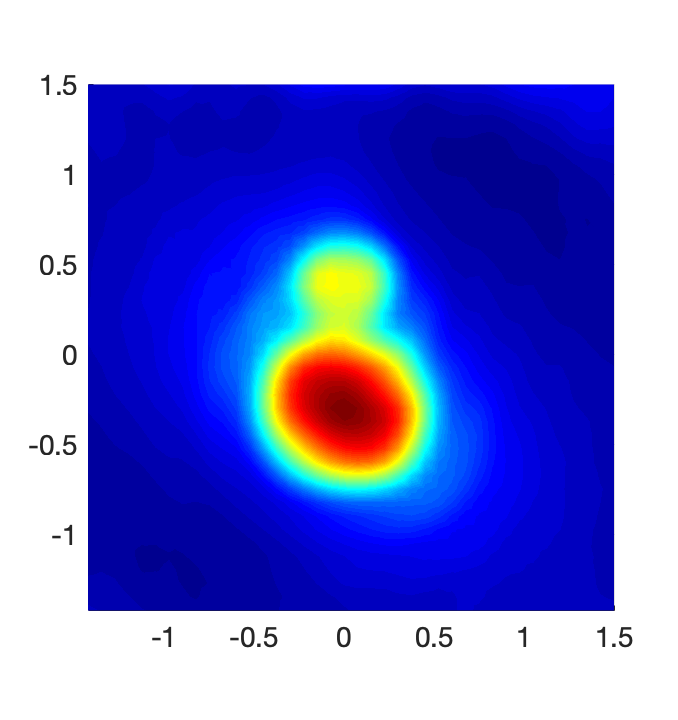}} \hspace{0.cm}
\subfloat[]{\includegraphics[width=5cm]{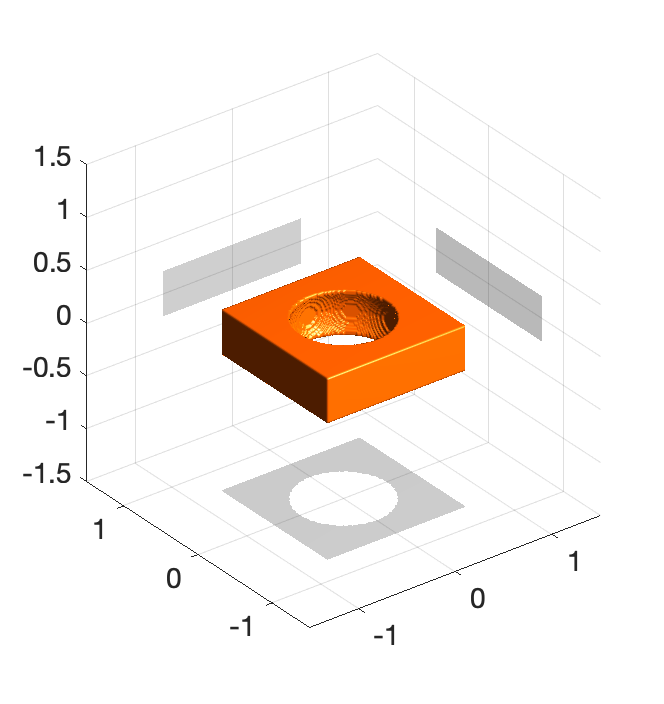}} \hspace{0.cm}
\subfloat[]{\includegraphics[width=5cm]{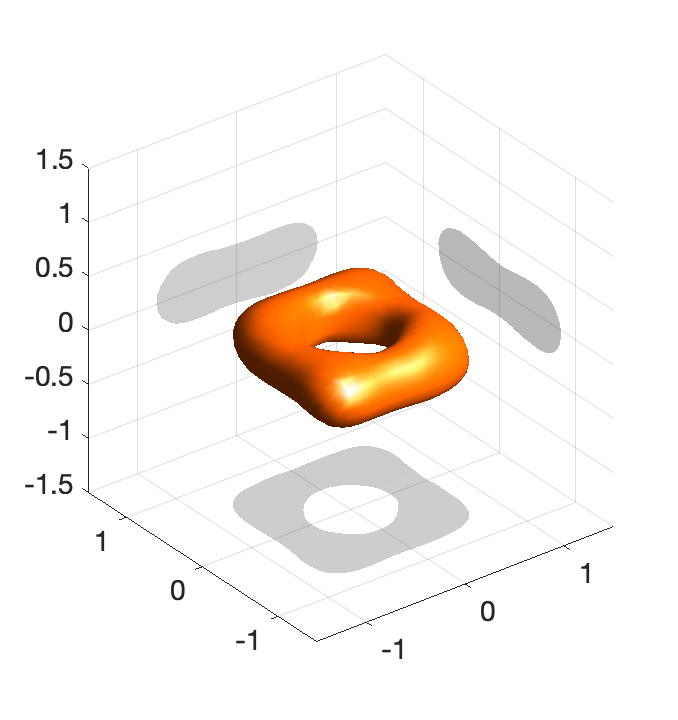}}  \hspace{0.cm}
\subfloat[]{\includegraphics[width=4.5cm]{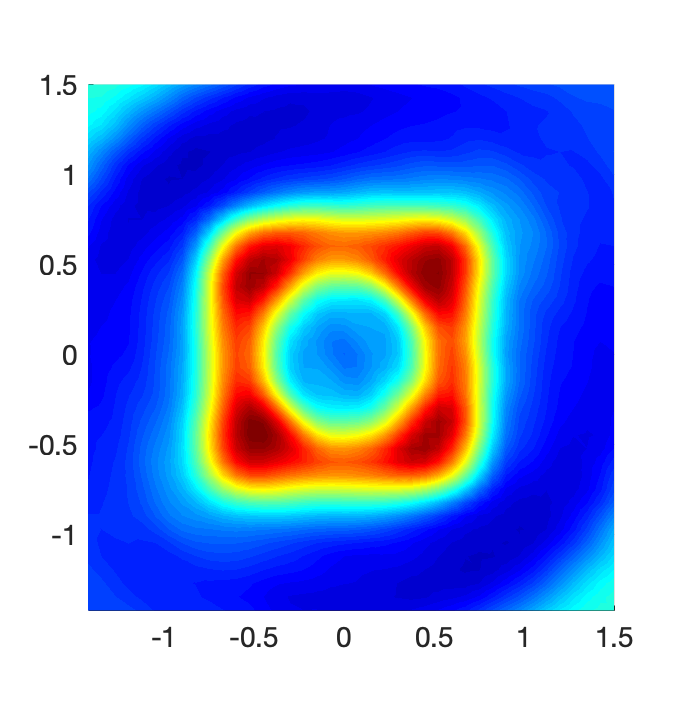}} \hspace{0.cm}
\subfloat[]{\includegraphics[width=5cm]{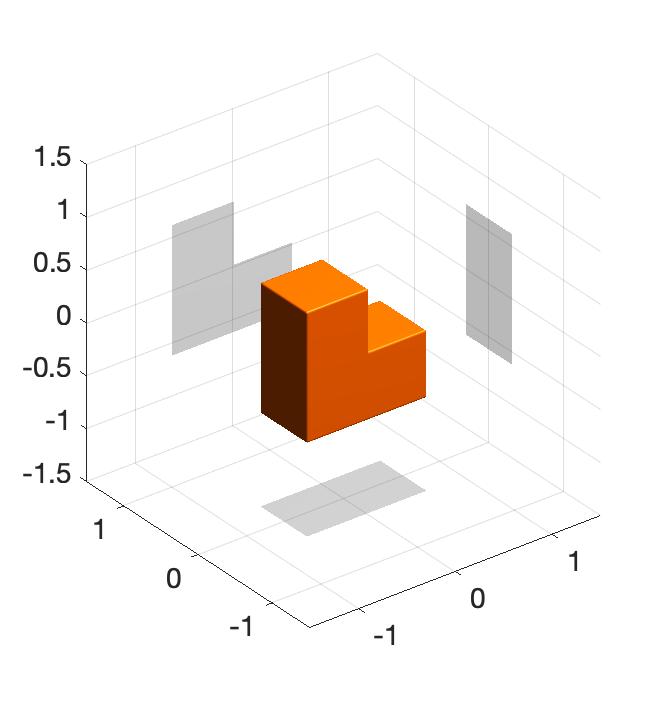}} 
\vspace{-0.cm}
\subfloat[]{\includegraphics[width=5cm]{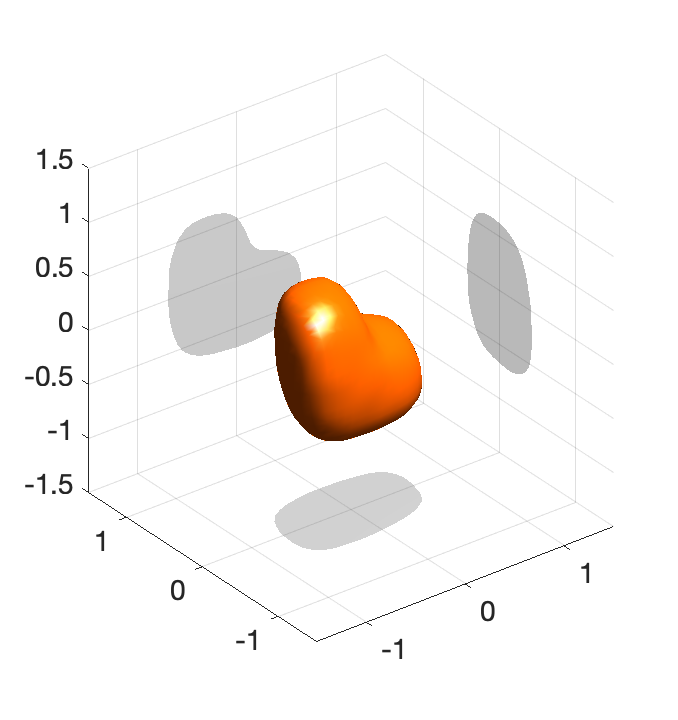}} 
\vspace{-0.cm}
\subfloat[]{\includegraphics[width=4.5cm]{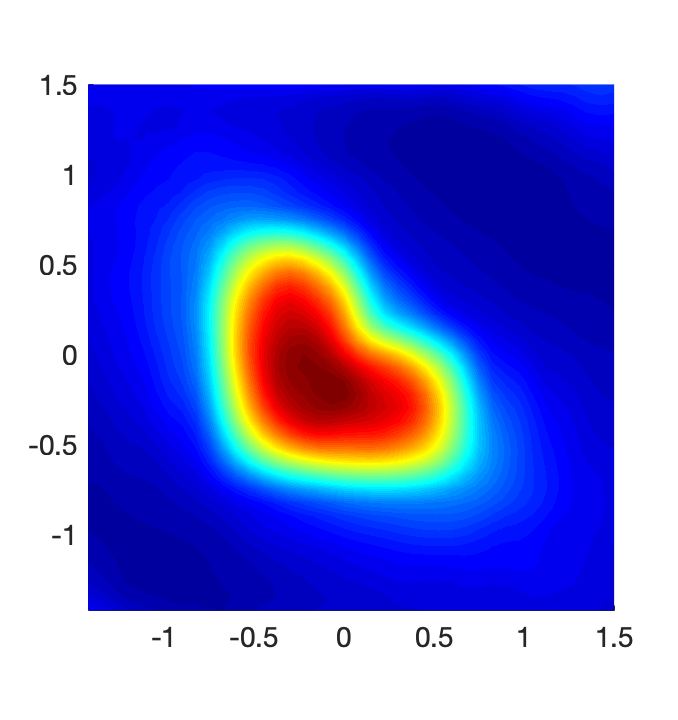}} 
\caption{Reconstruction of scatterers with more complicated shapes using OSM.  There is 
30$\%$ noise in the scattering data,  $N_{\widehat{\x}}\times N_\di = 325\times 325$ and $k = 12$.  
The first column is the exact geometries and the second column is
 the corresponding reconstructions.  
Cross-sectional views of the computed $\mathcal{I}_{OSM}$ are displayed in the last column. 
The isovalues in  the isosurface plotting are respectively $1/3$ and $1/2$ of the maximal value of the computed 
 $\mathcal{I}_{OSM}$ for (b) and (e, h).
}
\label{fi4}
\end{figure}

\section{Summary}
We study the OSM for solving the  electromagnetic  inverse scattering problem for  anisotropic media with far field data. 
We propose an imaging functional for the OSM which is able to  compute the location and shape of 
electromagnetic scatterers in a fast and robust way. Using tools of the Factorization method analysis and 
the Funk-Hecke formula 
we are able to establish a rigorous justification and resolution analysis for 
the proposed imaging functional. We also prove that this  functional is equivalent to that of the DSM 
and that our resolution analysis for the OSM can be directly applied to the DSM.  Numerical results 
for three-dimensional anisotropic scatterers are presented. Together with our recent work 
in~\cite{Nguye2019}, where the OSM is justified for the electromagnetic inverse scattering with one incident plane wave, 
we have provided a  versatile approach for solving the electromagnetic inverse scattering problem. \\

{\bf Acknowledgement}. The work of DLN is partially supported by NSF grant DMS-1812693.

\end{document}